\documentclass[12pt]{article}
\usepackage{amsthm,wasysym,amssymb,eufrak,indentfirst,color,graphicx,cite,mathrsfs}
\begin{document}
\baselineskip 17pt
\newtheorem{theorem}{Theorem}[section]
\newtheorem{lemma}[theorem]{Lemma}
\newtheorem{proposition}[theorem]{Proposition}
\newtheorem{example}[theorem]{Example}
\newtheorem{definition}[theorem]{Definition}
\newtheorem{corollary}[theorem]{Corollary}
\newtheorem{remark}[theorem]{Remark}
\newtheorem{problem}[theorem]{Problem}
\title{\textbf{On supersolubility of finite groups admitting a Frobenius group of automorphisms with fixed-point-free kernel}\thanks{The research is supported by an NNSF grant of China (grant No. 11371335) and Wu Wen-Tsun Key Laboratory of Mathematics, USTC, Chinese Academy of Science. The second author is also supported by the Start-up Scientific Research Foundation of Nanjing Normal University (grant No. 2015101XGQ0105) and a project funded by the Priority Academic Program Development of Jiangsu Higher Education Institutions.}}
\author{{Xingzheng Tang$^{1}$, Xiaoyu Chen$^{2,}$\thanks{Corresponding author.}, Wenbin Guo$^{1}$}\\
{\small $^1$School of Mathematical Sciences, University of Science and Technology of China,}\\
{\small Hefei 230026, P. R. China}\\
{\small $^2$School of Mathematical Sciences and Institute of Mathematics, Nanjing Normal University,}\\
{\small Nanjing 210023, P. R. China}\\
{\small E-mails: tangxz@mail.ustc.edu.cn, jelly@njnu.edu.cn, wbguo@ustc.edu.cn}}

\date{}
\maketitle
\begin{abstract}
Assume that a finite group $G$ admits a Frobenius group of automorphisms $FH$ with kernel $F$ and complement $H$ such that $C_{G}(F)=1$. In this paper, we investigate this situation and prove that if $C_G(H)$ is supersoluble and $C_{G'}(H)$ is nilpotent, then $G$ is supersoluble. Also, we show that $G$ is a Sylow tower group of a certain type if $C_{G}(H)$ is a Sylow tower group of the same type.
\end{abstract}
\renewcommand{\thefootnote}{\empty}

\section{Introduction}

Throughout this paper, all groups mentioned are assumed to be finite. $G$ always denotes a group, $p$ denotes a prime, $\pi$ denotes a set of primes, and $\mathbb{P}$ denotes the set of all primes. For any group $G$, we use the symbol $\pi(G)$ to denote the set of prime divisors of $|G|$.

Recall that a Frobenius group $FH$ with kernel $F$ and complement $H$ can be characterized as a group which is a semidirect product of a normal subgroup $F$ by $H$ such that $C_F (h) = 1$ for every non-identity element $h$ of $H$. Recently, much research was focused on the case when a Frobenius group $FH$ acts on a group $G$ such that $F$ acts fixed-point-freely, that is, $C_G(F)=1$. It was shown that various properties of $G$ are close to the corresponding properties of $C_G(H)$ in this situation, see \cite{K4,K5,K3,Shu,Shu1,Mak,K2,KMS}. For instance, E. I. Khukhro, N. Y. Makarenko and P. Shumyatsky \cite{KMS} proved that the rank of $G$ is bounded in terms of $|H|$ and the rank of $C_G(H)$, and $G$ is nilpotent if $C_G(H)$ is nilpotent; E. I. Khukhro \cite{K2} established that the Fitting height of $G$ is equal to the Fitting height of $C_G(H)$; P. Shumyatsky \cite{Shu} proved that if $F$ is cyclic and $C_G(H)$ satisfies a positive law of degree $k$, then $G$ satisfies a positive law of degree that is bounded solely in terms of $k$ and $|FH|$.

The main aim of this paper is to discuss the problem which was proposed by E. I. Khukhro in Fourth Group Theory Conference of Iran (see \cite{K1}). In the above situation, he asked that if $C_G(H)$ is supersoluble, whether a group $G$ is supersoluble or not. Though this problem has not been solved, we can give a positive answer if we suppose further that $C_{G'}(H)$ is nilpotent. In fact, a more generalized result is obtained. In Section 3, we prove that $G$ is $p$-supersoluble if $C_G(H)$ is $p$-supersoluble and $C_{G'}(H)$ is $p$-nilpotent. Moreover, we show that $G$ is a Sylow tower group of a certain type if $C_{G}(H)$ is a Sylow tower group of the same type.

\section{Preliminaries}

The following results are useful in our proof.

\begin{lemma}\label{2.1}(see \cite[Theorem 0.11]{BB}.)  Suppose that a group $G$ admits a nilpotent group of automorphisms $F$ such that $C_{G}(F)=1$. Then $G$ is soluble. \end{lemma}

\begin{lemma}\label{2.2}(see \cite[Lemma 2.2]{KMS}.) Let $G$ be a group admitting a nilpotent group of automorphisms $F$ such that $C_{G}(F)=1$. If $N$ is an $F$-invariant normal subgroup of $G$, then $C_{G/N} (F)=1$. \end{lemma}

\begin{lemma}\label{2.3}(see \cite[Lemma 2.3]{KMS}.) Suppose that a group $G$ admits a Frobenius group of automorphisms $FH$ with kernel $F$ and complement $H$. If $N$ is an $FH$-invariant normal subgroup of $G$ such that $C_{N}(F)$=1, then $C_{G/N}(H)$= $C_{G}(H)N/N$. \end{lemma}

\begin{lemma}\label{2.4} (see \cite[Lemma 2.2]{K2}.)  Let $FH$ be a Frobenius group with kernel $F$ and complement $H$. In any action of $FH$ with nontrivial action of $F$, the complement $H$ acts faithfully. \end{lemma}

\begin{lemma}\label{2.5}  Let $G$ be a non-trivial group admitting a Frobenius group of actions $FH$ with kernel $F$ and complement $H$ and $K$ be the kernel of $FH$ acts on $G$. If $C_{G} (F)= 1$, then $K<F$ and $FH/K$ is a Frobenius group with kernel $F/K$ and complement $HK/K$. \end{lemma}

\begin{proof}
As $F\nleq K$, we have that $H \cap K = 1$ by Lemma \ref{2.4}. Hence $K\leq F$ because $(|F|,|H|)=1$. Then for every non-trivial element $h\in H$, since $(|F|,|\langle h \rangle|)=1$, $C_{F/K}(h)=C_F(h)K/K=1$. This implies that $FH/K$ is a Frobenius group with kernel $F/K$ and complement $HK/K$.
\end{proof}

Recall that for a soluble group $G$, the Fitting series starts with $F_0(G) = 1$, followed by the Fitting subgroup $F_1(G) = F (G)$, and $F_{i+1}(G)$ is defined as the inverse image of $F (G/F_i (G))$. The next lemma is a collection of \cite[Theorem 2.1 and Corollary 4.1]{K2} and \cite[Lemma 2.4 and Theorem 2.7]{KMS}.

\begin{lemma}\label{2.7} Suppose that a group $G$ admits a Frobenius group of automorphisms $FH$ with kernel $F$ and complement $H$ such that $C_{G}(F)=1$. Then:

\smallskip
\textup{(1)} $|G|=|C_G(H)|^{|H|}$.
\smallskip

\textup{(2)} $G=\langle C_G(H)^f\mid f\in F\rangle$.
\smallskip

\textup{(3)} If $C_{G}(H)$ is nilpotent, then $G$ is nilpotent.
\smallskip

\textup{(4)} $F_i(C_G(H))=F_i(G)\cap C_G(H)$.
\smallskip

\textup{(5)} $O_{\pi}(C_G(H))=O_{\pi}(G)\cap C_G(H)$ for any set of primes $\pi$.

\end{lemma}

In the following lemma, the symbols $\mathfrak{U}$ and $\mathfrak{A}(p-1)$ denote the class of all supersoluble groups and the class of all abelian groups of exponent dividing $p-1$, respectively. Also, a normal subgroup $N$ of $G$ is called $\mathfrak{U}$-hypercentral in $G$ if either $N=1$ or $N>1$ and all chief factors of $G$ below $N$ are cyclic. Let $Z_{\mathfrak{U}}(G)$ denote the $\mathfrak{U}$-hypercentre of $G$, that is, the product of all $\mathfrak{U}$-hypercentral normal subgroups of $G$.

\begin{lemma}\label{2.10} Let $E$ be a normal $p$-subgroup of a group $G$. Then $E\leq Z_{\mathfrak{U}}(G)$ if and only if $(G/C_G(E))^{\mathfrak{A}(p-1)}\leq O_p(G/C_G(E))$.\end{lemma}

\begin{proof}
The necessity directly follows from \cite[Lemma 2.2]{Ski}. Now we prove the sufficiency. Since $O_p(G/C_G(H/K))=1$ for any chief factor $H/K$ of $G$ below $E$ and $C_G(E)\leq C_G(H/K)$, $G/C_G(H/K)\leq \mathfrak{A}(p-1)$. Hence $|H/K|=p$ by \cite[Lemma 2.1]{Ski}. This shows that $E\leq Z_{\mathfrak{U}}(G)$.
\end{proof}

\section{Main Results}

Firstly, we begin to show the connection between the properties of $G$ and $C_G(H)$ by proving that $G$ is $p$-closed (resp. $p$-nilpotent) if $C_G(H)$ is $p$-closed (resp. $p$-nilpotent).

\begin{theorem}\label{T1}
Suppose that a group $G$ admits a Frobenius group of automorphisms $FH$ with kernel $F$ and complement $H$ such that $C_{G}(F)=1$. If $C_{G}(H)$ is $p$-closed, then $G$ is $p$-closed.
\end{theorem}

\begin{proof} Suppose that the theorem is not true. Then we may consider a counterexample $G$ of minimal order. We proceed the proof via the following steps:

\smallskip

(1) \textit{$F(G)=O_{q}(G)$, where $q$ is a prime such that $q\neq p$, and $G/O_q(G)$ is $p$-closed.}

\smallskip

By Lemma \ref{2.1}, $G$ is soluble, and so $F(G) \neq 1$. Let $q$ be any prime dividing $|F(G)|$. Then $G/O_{q}(G)$ is $FH$-invariant. In view of Lemmas \ref{2.2}, \ref{2.3} and \ref{2.5}, it is easy to see that the hypothesis of the theorem still holds for $G/O_{q}(G)$. By the minimality of our counterexample, we have that $G/O_{q}(G)$ is $p$-closed. If $q=p$, then $G$ is $p$-closed, a contradiction. Thus $q\neq p$. Now assume that there exists another prime $r$ dividing $|F(G)|$. Then with the same argument as above, $G/O_{r}(G)$ is $p$-closed, and consequently, $G$ is $p$-closed. This contradiction shows that $F(G)=O_{q}(G)$.

\smallskip

(2) \textit{$G=O_{q}(G)P$, where $P$ is a Sylow $p$-subgroup of $G$.}

\smallskip

By (1), $O_{q}(G)P$ is an $FH$-invariant normal subgroup of $G$. In view of Lemma \ref{2.5}, $O_{q}(G)P$ satisfies the hypothesis of the theorem. If $O_{q}(G)P<G$, then by the minimality of our counterexample, we have that $O_{q}(G)P$ is $p$-closed, and so $G$ is $p$-closed, a contradiction. Hence $G=O_{q}(G)P$.

\smallskip

(3) \textit{The final contradiction.}

\smallskip

Obviously, $G$ is $q$-closed by (2), and so $C_{G}(H)$ is nilpotent. Then by Lemma \ref{2.7}(3), $G$ is nilpotent. The final contradiction finishes the proof.
\end{proof}

\begin{theorem}\label{T5}
Suppose that a group $G$ admits a Frobenius group of automorphisms $FH$ with kernel $F$ and complement $H$ such that $C_{G}(F)=1$. If $C_{G}(H)$ is $p$-nilpotent, then $G$ is $p$-nilpotent.
\end{theorem}

\begin{proof}
Suppose that the theorem does not hold. Let $G$ be a counterexample of minimal order. Then:

\smallskip

(1) \textit{$O_{p'}(G)= 1$, and so $F(G)=O_p(G)$}.

\smallskip

By Lemmas \ref{2.2}, \ref{2.3} and \ref{2.5}, $G/O_{p'}(G)$ satisfies the hypothesis of the theorem. If $O_{p'}(G)\neq 1$, then by our choice, $G/O_{p'}(G)$ is $p$-nilpotent, and thereby $G$ is $p$-nilpotent, a contradiction. Thus $O_{p'}(G)=1$, and so $F(G)=O_p(G)$.

\smallskip

(2) \textit{$G=O_{p}(G)Q$, where $Q$ is the unique $FH$-invariant Sylow $q$-subgroup of $G$ with $q\neq p$}.

\smallskip

By \cite[Lemma 2.6]{KMS}, there exists a unique $FH$-invariant Sylow $q$-subgroup of $G$, denoted by $Q$. If $O_p(G)Q<G$, then $O_p(G)Q$ satisfies the hypothesis of the theorem by Lemma \ref{2.5}, and so $O_p(G)Q$ is $p$-nilpotent by the minimality of our counterexample. This yields that $O_p(G)Q$ is nilpotent. Then $Q\leq C_G(O_p(G))$. Since $G$ is soluble by Lemma \ref{2.1}, $C_G(F(G))\leq F(G)$. This implies that $Q\leq F(G)$, which contradicts (1). Hence $G=O_{p}(G)Q$.

\smallskip

(3) \textit{The final contradiction.}

\smallskip

Since $C_G(H)$ is $p$-nilpotent and $G$ is $p$-closed by (2), we have that $C_G(H)$ is nilpotent, which forces that $G$ is also nilpotent by Lemma \ref{2.7}(3). This is the final contradiction.
\end{proof}

The following lemma can be viewed as not only an improvement of \cite[Lemma 2.6]{KMS}, but also a key step in the proof of Theorem \ref{T3}.

\begin{lemma}\label{2.6} Suppose that a group $G$ admits a Frobenius group of automorphisms with kernel $F$ and complement $H$ such that $C_{G} (F)= 1$. Then for any subset of primes $\pi$ of $\pi (G)$, there exists a unique $FH$-invariant Hall $\pi$-subgroup of $G$. Furthermore, the set of all unique $FH$-invariant Hall subgroups forms a Hall system of $G$. \end{lemma}

\begin{proof} By Lemma \ref{2.1}, $G$ is soluble, and so is $GF$. Since $C_{G}(F) =1$, it is easy to see that $F$ is a Carter subgroup of $GF$. Then $F$ contains a system normalizer of $GF$ by \cite[Chapter V, Theorem 4.1]{Doe}. By \cite[Chapter I, Theorem 5.6]{Doe}, a system normalizer covers all central chief factors of $GF$, and so $F$ is a system normalizer of $GF$ because $F$ is nilpotent. This implies that there exists an $F$-invariant Hall $\pi$-subgroup $S$ of $G$. If $S$ and $S^g$ are both $F$-invariant Hall $\pi$-subgroups of $G$, where $g\in G$, then $F$ and $F^{g^{-1}}$ are two Carter subgroups of $N_{GF}(S)=N_{G}(S)F$. Hence $F=F^{g^{-1}g'}$ for some $g'\in N_{G}(S)$. Since $N_{G}(F)= C_{G}(F)= 1$, $g^{-1}g'=1$, and so $g\in N_{G}(S)$. Thus $S=S^{g}$. This shows that $S$ is the unique $F$-invariant Hall $\pi$-subgroup of $G$. Since $S^h$ is $F$-invariant for any $h\in H$, we have that $S$ is also $H$-invariant. Furthermore, let $\pi(G)=\{p_1,\dots,p_r\}$ and $S_i$ be the unique $FH$-invariant Sylow $p_i$-complement of $G$ for $1\leq i\leq r$. Note that for every subset of primes $\pi$ of $\pi (G)$, $G_\pi=\bigcap\{S_i:p_i\in \pi(G)\backslash \pi\}$ is the unique $FH$-invariant Hall $\pi$-subgroup of $G$. Then by \cite[Chapter I, Proposition 4.4]{Doe}, the set of all unique $FH$-invariant Hall subgroups forms a Hall system of $G$.\end{proof}

Now we can establish our main result as follows.

\begin{theorem}\label{T3} Suppose that a group $G$ admits a Frobenius group of automorphisms $FH$ with kernel $F$ and complement $H$ such that $C_{G}(F) = 1$. If $C_{G}(H)$ is $p$-supersoluble and $C_{G'}(H)$ is $p$-nilpotent, then $G$ is $p$-supersoluble.
\end{theorem}

\begin{proof} Suppose that the theorem is not true. Then we may consider a counterexample $G$ of minimal order. Then:

\smallskip

(1) \textit{Let $S$ be an $FH$-invariant proper subgroup of $G$ and $N$ be a non-trivial $FH$-invariant normal subgroup of $G$. Then $S$ and $G/N$ are both $p$-supersoluble.}

\smallskip

By the minimality of our counterexample, it is sufficient to prove that $S$ and $G/N$ both satisfy the hypothesis of the theorem. Clearly, $S$ satisfies the hypothesis of the theorem by Lemma \ref{2.5}. Let $C_{G'N/N}(H)=A/N$ and $C_{G'/G'\cap N}(H)=B/G'\cap N$. Note that by Lemma \ref{2.3}, $A/N=BN/N=C_{G'}(H)N/N$ is $p$-nilpotent and $C_{G/N}(H)=C_G(H)N/N$ is $p$-supersoluble.
Hence by Lemmas \ref{2.2} and \ref{2.5}, $G/N$ also satisfies the hypothesis of the theorem.
\smallskip

(2) \textit{$O_{p'}(G)=O_{p'}(C_G(H))=1$.}

\smallskip

Suppose that $O_{p'}(G)\neq 1$. Then by (1), $G/O_{p'}(G)$ is $p$-supersoluble. Thus $G$ is $p$-supersoluble. This contradiction shows that $O_{p'}(G)=1$, which forces that $O_{p'}(C_G(H))=1$ by Lemma \ref{2.7}(5).
\smallskip

(3) \textit{$F_p(G)=F(G)=P$ is the socle of $G$, where $F_p(G)$ denotes the largest normal $p$-nilpotent subgroup of $G$ and $P$ denotes the normal Sylow $p$-subgroup of $G$, and $C_G(P)=P$.}

\smallskip

Since $C_G(H)$ is $p$-supersoluble and $O_{p'}(C_G(H))=1$ by (2), $C_G(H)$ is $p$-closed by \cite[Lemma 2.1.6]{Bal}. It follows from Theorem \ref{T1} that $G$ is $p$-closed. As $O_{p'}(G)=1$ by (2), $F_p(G)=F(G)=P$, where $P$ is the normal Sylow $p$-subgroup of $G$. If $\Phi(G)\neq 1$, then $G/\Phi(G)$ is $p$-supersoluble by (1). This implies that $G$ is $p$-supersoluble, which is impossible. Thus $\Phi(G)=1$, and so $P$ is the socle of $G$. Note that by Lemma \ref{2.1}, $G$ is soluble. It follows that $C_G(P)\leq P$. Therefore, we have that $C_G(P)=P$.

\smallskip

(4) \textit{$G$ has the unique $FH$-invariant Hall $p'$-subgroup $T$ and $T$ is abelian.}

\smallskip

By Lemma \ref{2.5}, $G'$ satisfies the hypothesis of Theorem \ref{T5}. Thus $G'$ is $p$-nilpotent, and so $G'\leq P$ by (3). This implies that $G/P$ is abelian. In view of Lemma \ref{2.6}, $G$ has the unique $FH$-invariant Hall $p'$-subgroup $T$, and clearly, $T$ is abelian.

\smallskip

(5) \textit{The exponent of $C_{T}(H)$ divides $p-1$.}

\smallskip

Let $C=P\cap C_G(H)$. Then obviously, $C$ is the normal Sylow $p$-subgroup of $C_G(H)$. By (3) and Lemma \ref{2.7}(4), $F(C_G(H))=C$, and so $C_{C_G(H)}(C)=C$. Since $C_G(H)$ is $p$-supersoluble, $C\leq Z_{\mathfrak{U}}(C_G(H))$, and consequently, $C_G(H)/C=C_G(H)/C_{C_G(H)}(C)\in \mathfrak{A}(p-1)$ by Lemma \ref{2.10}. This implies that the exponent of $C_{T}(H)$ divides $p-1$.

\smallskip

(6) \textit{The final contradiction.}

\smallskip

By applying Lemmas \ref{2.5} and \ref{2.7}(2) for $T$, we have that $T = \langle C_{T}(H)^{f}| f \in F\rangle$. Since $T$ is abelian by (4), the exponent of $C_{T}(H)$ is equal to the exponent of $T$. Then by (5), the exponent of $T$ divides $p-1$, and so $T\in \mathfrak{A}(p-1)$ by (4). Hence by (3), $G/C_G(P)=G/P\cong T \in \mathfrak{A}(p-1)$, which yields that $P\leq Z_{\mathfrak{U}}(G)$ by Lemma \ref{2.10}. Thus $G$ is $p$-supersoluble. The final contradiction finishes the proof.
\end{proof}

From Theorem \ref{T3}, we can directly deduce the next corollary.

\begin{corollary}
Suppose that a group $G$ admits a Frobenius group of automorphisms $FH$ with kernel $F$ and complement $H$ such that $C_{G}(F) = 1$. If $C_{G'}(H)$ is nilpotent and $C_{G}(H)$ is supersoluble, then $G$ is supersoluble.
\end{corollary}

Recall that if $\sigma$ denotes a linear ordering on $\mathbb{P}$, then a group $G$ is called a Sylow tower group of type $\sigma$ if there exists a series of normal subgroups of $G$: $1=G_0\leq G_1\leq\cdots\leq G_n=G$ such that $G_i/G_{i-1}$ is a Sylow $p_i$-subgroup of $G/G_{i-1}$ for $1\leq i\leq n$, where $p_1\prec p_2\prec\cdots\prec p_n$ is the ordering induced by $\sigma$ on the distinct prime divisors of $|G|$. Here we arrive at the next theorem.

\begin{theorem}\label{T2} Suppose that a group $G$ admits a Frobenius group of automorphisms $FH$ with kernel $F$ and complement $H$ such that $C_{G}(F)= 1$. If $C_{G}(H)$ is a Sylow tower group of a certain type, then $G$ is a Sylow tower group of the same type.
\end{theorem}

\begin{proof} Suppose that $C_{G}(H)$ is a Sylow tower group of type $\sigma$ and $p_1\prec p_2\prec\cdots\prec p_r$ is the ordering induced by $\sigma$ on the distinct prime divisors of $|G|$. Then by Lemma \ref{2.7}(1), $p_i$ divides $|C_G(H)|$ for $1\leq i\leq r$. Since $C_{G}(H)$ is a Sylow tower group of type $\sigma$, $C_{G}(H)$ is $p_1$-closed, and so $G$ is $p_1$-closed by Theorem \ref{T1}. Let $G_1$ be the normal Sylow $p_1$-subgroup of $G$. Then clearly, $G/G_1$ is $FH$-invariant. Since $G/G_1$ satisfies the hypothesis of the theorem by Lemmas \ref{2.2}, \ref{2.3} and \ref{2.5}, by induction, $G/G_1$ is a Sylow tower group of type $\sigma$, and so is $G$.
\end{proof}

\end{document}